\documentclass[12pt, a4paper]{amsart}

\usepackage{amsmath,amssymb}

\usepackage[colorlinks=true, pdfstartview=FitV, linkcolor=blue, citecolor=blue, urlcolor=blue]{hyperref}
\usepackage{color}
\usepackage{enumitem}
\usepackage{setspace}
\usepackage{tikz}
\usepackage{graphicx}

\newtheorem{theorem}{Theorem}[section]
\newtheorem{lemma}[theorem]{Lemma}
\newtheorem{proposition}[theorem]{Proposition}
\newtheorem{corollary}[theorem]{Corollary}

\theoremstyle{definition}

\theoremstyle{remark}
\newtheorem{remark}[theorem]{Remark}
\numberwithin{equation}{section}

\newcommand{\R}{\mathbb{R}}

\newcommand{\rn}{\mathbb{R}^n}

\begin{document}

\title[Operators on weighted local  Morrey spaces]
{Singular and fractional integral operators on weighted local Morrey spaces}

\author{Javier Duoandikoetxea and Marcel Rosenthal} 
\address{Sopela, Basque Country, Spain}

\email{javier.duoandikoetxea@ehu.eus}

\address{Stuttgart, Germany} \email{marcel.rosenthal@uni-jena.de} 


\begin{abstract}
We obtain a characterization of the weighted inequalities for the Riesz transforms on weighted local Morrey spaces. The condition is sufficient for the boundedness on the same spaces of all Calder\'on-Zygmund operators suitably defined on the functions of the space. In the case of the fractional maximal operator and the fractional integral we obtain a characterization valid for exponents satisfying the Sobolev relation. For power weights we get sharp results in a larger range of exponents for the usual versions of weighted Morrey spaces.  
\end{abstract}

\maketitle

\section{Introduction}\label{intro}

 We define for $1\le p<\infty$ the weighted local (also known as central) Morrey space $\mathcal {LM}^{p}(\varphi,w)$ as the class of measurable functions in $\rn$ for which  
\begin{equation}\label{normdef}
\|f\|_{\mathcal {LM}^{p}(\varphi,w)} := \sup_{B}\left(\frac 1{\varphi(B)}\int_{B}|f|^p w\right)^{1/p}<\infty,
\end{equation}
where the supremum is taken over all the Euclidean balls in $\rn$ centered at $0$. Here $w$ is a weight, that is,  a nonnegative locally integrable function, and $\varphi$ is a function defined on the set of balls in $\rn$ centered at $0$, with values in $(0,\infty)$. For the global Morrey spaces all the Euclidean balls in $\rn$ are taken.

The most studied cases of $\varphi$ correspond to the weighted Morrey spaces introduced by N.~Samko in \cite{Sam09}, and by Y.~Komori and S.~Shirai in \cite{KS09}. In general, we need to impose some conditions on $\varphi$ to avoid trivial spaces. We fix some conditions in Section \ref{bi}. 

The characterization of the weights $w$ for which the Hardy-Littlewood maximal operator is bounded on global Morrey spaces is an open problem even for the most familiar cases of $\varphi$ mentioned in the preceding paragraph. Nevertheless, a remarkable result of 
S.~Nakamura, Y.~Sawano, and H.~Tanaka in \cite{NST19}  gives necessary and sufficient conditions for the boundedness of the Hardy-Littlewood maximal operator $M$ on weighted local Morrey spaces. In \cite{DR21} the authors of this paper obtained  a simplified characterization of the same result in the form of a Muckenhoupt-type condition adapted to the Morrey setting. More precisely, we showed that the condition   
\begin{equation}\label{apmdef}
	\sup_B \frac{\|\chi _B\|_{\mathcal {LM}^{p}(\varphi,w)}\|\chi _B\|_{\mathcal {LM}^{p}(\varphi,w)'}}{|B|} <\infty
\end{equation}
is necessary and sufficient for the boundedness of $M$ on $\mathcal {LM}^{p}(\varphi,w)$. (The notation $X'$ stands for the K\"othe dual of the Banach lattice $X$. See Subsection \ref{kothesub}.) In this paper we deal first with singular integral operators and show that the slightly stronger condition 
\begin{equation}\label{sicond}
	\sup_B \frac{\|\chi _B\|_{\mathcal {LM}^{p}(\varphi,w)}\|M\chi _B\|_{\mathcal {LM}^{p}(\varphi,w)'}}{|B|} <\infty
\end{equation}
is necessary and sufficient for the boundedness of the Riesz transforms on $\mathcal {LM}^{p}(\varphi,w)$. This is valid for $1<p<\infty$, and for $p=1$ the result holds if we consider weak-type inequalities. The sufficiency is extended to all Calder\'on-Zygmund operators. We remark that due to the lack of density of smooth functions in the Morrey-type spaces the operators need to be defined in a convenient way. Condition \eqref{sicond} is the limiting case of the sufficient condition obtained in \cite[Theorem 7.3]{DR21} by using extrapolation of Lebesgue-weighted inequalities.

Next we work with the fractional maximal operator and the fractional integrals. In this case our results are limited to a particular form of the spaces. In the classical paper \cite{MW74}, B.~Muckenhoupt and R.~Wheeden characterized the weights $w$ such that the fractional integral $I_\alpha$ is bounded from $L^p(w^p)$ to $L^q(w^q)$, when $p$ and $q$ are related by the Sobolev condition $1/p-1/q=\alpha/n$. The characterizations we obtain for the boundedness of the fractional operators on weighted local Morrey spaces are also limited by the same condition on the exponents and similar conditions on the weights. More precisely, we obtain necessary and sufficent conditions for the boundedness of $M_\alpha$ and $I_\alpha$ from $\mathcal {LM}^{p}(\varphi^p,w^p)$ to $\mathcal {LM}^{q}(\varphi^q,w^q)$, with $p$ and $q$ as above. They are given by expressions which are similar to \eqref{apmdef} and \eqref{sicond}, respectively. In the case of $M_\alpha$ a necessary and sufficient condition was already obtained in \cite{NST19}. As we did for the Hardy-Littlewood maximal operator, we simplify the characterization and give it by means of a single condition, namely,    
 \begin{equation}\label{suffracmax1}
\sup_B \frac{\|\chi _B\|_{L\mathcal M^{q}(\varphi^q,w^q)}\|\chi _B\|_{L\mathcal M^{p}(\varphi^p,w^p)'}}{|B|^{1-\alpha/n}} <\infty.
\end{equation}
For the characterization of the boundedness of $I_\alpha$ we need a stronger condition as happened for singular integrals. In this case the condition reads
\begin{equation}\label{suffracint1}
\sup_B \frac{\|\chi _B\|_{L\mathcal M^{q}(\varphi^q,w^q)}\|(M\chi _B)^{1-\alpha/n}\|_{L\mathcal M^{p}(\varphi^p,w^p)'}}{|B|^{1-\alpha/n}} <\infty.
\end{equation}
A term of the form $(M\chi _B)^{1-\alpha/n}$ already appeared in the characterization of two-weighted Lebesgue inequalities for $M_\alpha$ and $I_\alpha$ (see \cite{SW92} and \cite{W93}).

It is known that in the case of unweighted global Morrey spaces the range of $p$ and $q$ can be extended to $1/p-1/q=\alpha/(n-\lambda)$ (where $\varphi(B)=r_B^\lambda$), due to a result of D.~R.~Adams. Weighted estimates on global Morrey spaces in the Adams range were obtained by T.~Iida,  Y.~Komori-Furuya, and E.~Sato in \cite{IKS11}, and more generally by S.~Nakamura, Y.~Sawano, and H.~Tanaka in \cite{NST18}. It is known that for unweighted local Morrey spaces the result in the Adams range does not hold (see \cite{KFS17}). We extend the impossibility of Adams-type estimates also to the weighted setting (Theorem \ref{teonecrange}).

In our approach in \cite{DR21} a crucial role was played by the local Hardy-Littlewood maximal operator, where local refers to the use of balls separated from the origin. These operators and their weighted inequalities were studied in \cite{LS10} and \cite{HSV14}. In this paper we also need local versions of the operators we study. The weighted results we use were given in \cite{HSV19} and are presented in Subsection  \ref{localsubsec}.

In Section \ref{zortzi} we consider power weights and obtain sharp results for the special cases of weighted local Morrey spaces of Samko and Komori-Shirai type.

\subsection*{Notation} \begin{itemize}
\item Given a set $A$, $\chi_A$ is its characteristic function, $|A|$ is its Lebesgue measure, and for a weight $v$, $v(A)$ is the integral of $v$ over $A$ (but notice that $\varphi$ is not a weight in the definition of the norm of the Morrey space);  
\item for a ball $B$, $r_B$ will denote its radius and $c_B$ its center; for $\lambda>0$, $\lambda B$ denotes the ball with center $c_B$ and radius $\lambda r_B$; $\widetilde{B}$ is the smallest ball centered at the origin containing $B$, that is, $\widetilde{B}=B(0, |c_B|+r_B)$;
\item For $\varphi$ defined on balls we denote as $\varphi^p$ the function such that $\varphi^p(B)=\varphi(B)^p$;
\item $F\lesssim G$ means that the inequality $F\le c\,G$ holds with a constant $c$ depending only on the ambient parameters; and $F\sim G$ means that both $F\lesssim G$ and $G\lesssim F$ hold.
\end{itemize}


\section{Preliminaries}\label{bi}

The most used versions of weighted Morrey spaces correspond to the choices $\varphi(B)=r_B^{\lambda}$ ($0<\lambda<n$), introduced by N.~Samko in \cite{Sam09}, and
$\varphi(B)=w(B)^{\lambda/n}$ ($0<\lambda<n$), introduced by Komori and Shirai in \cite{KS09}. Generalized weighted (global and local) Morrey spaces have been considered by several authors. For general weighted local Morrey spaces we refer for instance to \cite{NST19}. 

\subsection{Conditions on $\varphi$}
We assume that $\varphi$ is doubling, that is, 
\begin{equation*}
\varphi(2B)\lesssim \varphi(B)
\end{equation*}
for $B$ centered at the origin. 
We also assume that it satisfies a reverse doubling property of the form 
\begin{equation}\label{rdcond}
\frac {\varphi(B(0,r))}{\varphi(B(0,R))}\lesssim \left( \frac rR\right)^{\delta},
\end{equation}
for some $\delta>0$.  From this property it follows that if $0<b<1$ it holds that 
\begin{equation}\label{rdplus}
\sum_{j=0}^\infty \varphi(B(0,b^jR))\lesssim  \varphi(B(0,R)).
\end{equation}

For a measure $\mu$ the doubling condition on all the Euclidean balls of $\rn$ implies the reverse doubling property (see \cite[Lemma 20]{ST89}). When $\varphi$ is of the form $\varphi(B) =\mu(B)^\lambda$ for a doubling measure $\mu$, then \eqref{rdcond} is satisfied.

\begin{proposition}\label{newnorm}
Assume that $\varphi$ satisfies \eqref{rdplus}. Then 
the norm \eqref{normdef} is equivalent to 
\begin{equation}\label{normrestrict}
\sup_{B: r_B=\kappa |c_B|}\left(\frac 1{\varphi(\widetilde B)}\int_{B}|f|^p w\right)^{1/p}<\infty,
\end{equation}
where $\kappa\in (0,1)$ is fixed. 
\end{proposition}
We recall that $\widetilde{B}$ denotes the smallest ball centered at the origin containing $B$. The main idea for this reduction comes from \cite{NST19}. We adapt here to local Morrey spaces the proof given in \cite[Proposition 2.1]{DR20} for global Morrey spaces. 

\begin{proof}
It is clear that \eqref{normrestrict} is not greater than \eqref{normdef} because $B\subset \widetilde B$.

Let $b>1/3$. Set $A_j:= B(0, b^jR)\setminus B(0, b^{j+1}R)$. There exists $C(n)$ depending only on the dimension $n$ such that the annulus $A_j$ can be covered by $C(n)$ balls $B_{k,j}:=B(y_{k,j},r_j)$, $k=1,\dots,C(n)$, for which $|y_{k,j}|=b^j(1+b)R/2$ (that is, $y_{k,j}$ is equidistant from the spheres of the boundary of the annulus) and $r_j=b^j(1-b)R$ (the width of the annulus). By construction $\widetilde{B_{k,j}}$ is independent of $k$ and satisfies $B(0, b^jR)\subset \widetilde{B_{k,j}}\subset  B(0, b^{j-1}R)$. Then $\varphi(\widetilde{B_{k,j}})\lesssim \varphi(B(0, b^{j-1}R))$. Consequently,
 \begin{equation*}
\aligned
&\frac 1{\varphi(B(0, R))}\int_{B(0, R)}|f|^p w =\frac 1{\varphi(B(0, R))}\sum_{j=0}^\infty \int_{A_j}|f|^p w\\
&\qquad\le \frac 1{\varphi(B(0, R))}\,C(n)\sum_{j=0}^\infty \varphi(B(0, b^{j-1}R)) \sup_{k,j}\frac 1{\varphi(\widetilde{B_{k,j}})}\int_{B_{k,j}}|f|^p w.
\endaligned
\end{equation*}
The balls $B_{k,j}$ satisfy $r_B=\kappa |c_B|$ for $\kappa=2(1-b)/(1+b)$. Using \eqref{rdplus} we obtain the control of  \eqref{normdef} by \eqref{normrestrict}.
\end{proof}

\begin{remark}\label{rm22}
For a function $f$ supported in a ball $B$ with $r_B\le \kappa  |c_B|$ ($0<\kappa<1$) it holds that
\begin{equation}\label{locloc}
\|f\|_{\mathcal {LM}^{p}(\varphi,w)}\sim \left(\frac{1}{\varphi(\widetilde{B})}\int_B |f|^pw\right)^{1/p}.
\end{equation}
This is because the balls $B(0,R)$ with $R<|c_B|-r_B$ do not intersect $B$ and $\varphi(B(0,R))$ is essentially constant for $|c_B|-r_B<R<|c_B|+r_B$, due to the properties of $\varphi$.  In particular, if $f$ is the characteristic function of $B$,
\begin{equation}\label{locchar}
\|\chi_B\|_{\mathcal {LM}^{p}(\varphi,w)}\sim \left(\frac{w(B)}{\varphi(\widetilde{B})}\right)^{1/p}.
\end{equation} 
\end{remark}

\begin{proposition}\label{rm23}
Let $\gamma\in \R$. Define $\psi(B)=\varphi(B) r_B^\gamma$ and $v(x)=w(x) |x|^\gamma$. If both $\psi$ and $\varphi$ satisfy \eqref{rdplus}, then $\mathcal {LM}^{p}(\varphi,w)=\mathcal {LM}^{p}(\psi,v)$.
\end{proposition}

\begin{proof}
 It is enough to realize that for the balls $B$ involved in \eqref{normrestrict} if $x\in B$ we have $|x|\sim r_B$, and also $r_B\sim r_{\widetilde B}$.
\end{proof}

The spaces need not be the same without the reduction of the norm to \eqref{normrestrict}. For instance, if $\varphi(B)=r_B^\lambda$ with $\lambda>0$ and $w(x)\equiv 1$, the space $\mathcal {LM}^{p}(\varphi,w)$ is not trivial, but $\mathcal {LM}^{p}(\psi,v)=\{0\}$ if $\lambda+\gamma$ is negative.

\subsection{The K\"othe dual (associate space) of a Banach lattice} \label{kothesub}
Let $X$ be a Banach lattice of measurable functions in $\rn$ (a Banach space in which $|g|\le |f|$ a.e. and $f\in X$ implies $g\in X$ and $\|g\|_X\le \|f\|_X$). We define its K\"othe dual as the space $X'$ collecting the measurable functions $g$ such that
\begin{equation}\label{kothe}
\|g\|_{X'}:=\sup \left\{ \int_{\rn} |fg|: f\in X \text{ and } \|f\|_X\le 1\right\}<\infty.
\end{equation}
It is also a Banach lattice. An immediate consequence of \eqref{kothe} is the H\"older-type inequality
\begin{equation}\label{holder}
\int_{\rn} |fg| \leq \|f\|_X \|g\|_{X'}.
\end{equation}

\subsection{Weak-type local Morrey spaces}\label{subs23}

We define for $1\le p<\infty$ the weighted weak-type local Morrey space $W\mathcal {LM}^{p}(\varphi,w)$ as the collection of measurable functions in $\rn$ for which  
\begin{equation}\label{wnormdef}
\|f\|_{W\mathcal {LM}^{p}(\varphi,w)}  = \sup_{B}\sup_{t>0}\frac {t (w(\{x\in B: |f(x)|>t\})^{1/p}}{\varphi(B)^{1/p}}<\infty.
\end{equation}
It is clear that $\mathcal {LM}^{p}(\varphi,w)$ is continuously embedded in $W\mathcal {LM}^{p}(\varphi,w)$. Also that for a measurable set $E$, it holds
\begin{equation}\label{wcharac}
\|\chi_E\|_{W\mathcal {LM}^{p}(\varphi,w)} = \|\chi_E\|_{\mathcal {LM}^{p}(\varphi,w)}.
\end{equation}
In particular, we have \eqref{locchar} with $W\mathcal {LM}^{p}(\varphi,w)$ instead of $\mathcal {LM}^{p}(\varphi,w)$.

\subsection{The local Hardy-Littlewood maximal operator, local fractional operators and local singular integrals}\label{localsubsec}

For fixed $\kappa\in (0,1)$ we define the basis $\mathcal{B}_{\kappa,\text{loc}}$ as the family of balls $B$ such that $r_B<\kappa |c_B|$. For $\alpha\in [0,n)$ we define the operator $M_{\kappa,\alpha,\text{loc}}$ associated to $\mathcal{B}_{\kappa,\text{loc}}$ acting on $f$ at $x\ne 0$ as
\begin{equation}\label{defmaxloc}
M_{\kappa,\alpha,\text{loc}}f(x)=\sup_{x\in \mathcal{B}_{\kappa,\text{loc}}}\frac 1{|B|^{1-\alpha/n}}\int_{B}|f|.
\end{equation}
For $\alpha=0$ this is the local Hardy-Littlewood maximal operator, which we denote simply as $M_{\kappa,\text{loc}}$. For $\alpha\in (0,n)$ it is the local fractional maximal operator. Notice that the term local here is associated to a family of balls different from the one used in the definition of the local Morrey spaces.

C.-C.~Lin and K.~Stempak studied in \cite{LS10} a variant of $M_{\kappa,\text{loc}}$ and characterized its weighted inequalities by a local $A_p$ condition. E. Harboure, O. Salinas, and B. Viviani obtained in \cite{HSV14} a similar result in a more general setting, and in \cite{HSV19} they studied weighted inequalities for the fractional case and for singular integrals. Their result applied to the operator in \eqref{defmaxloc} says the following:  let $1<p<n/\alpha$ and $q$ given by $1/p-1/q=\alpha/n$. Then $M_{\kappa, \alpha, \text{loc}}$ is bounded from  $L^p(w^p)$ to $L^q(w^q)$ if and only if $w$ satisfies the $A_{p,q,\text{loc}}$ condition 
\begin{equation}\label{defapqloc}
\sup_{B\in \mathcal{B}_{\kappa,\text{loc}}} \frac{w^q(B)^{1/q}\,w^{-p'}(B)^{1/p'}}{|B|^{1-\frac \alpha n}}<\infty.
\end{equation}
For $\alpha=0$ this is the usual $A_p$ condition adapted to the local setting (written for $w^p$ instead of $w$), and for $\alpha>0$ this is the local version of the condition for fractional operators in \cite{MW74}. For $p=1$ and $q=n/(n-\alpha)$ the condition is modified as usual,  
\begin{equation}\label{defa1loc}
\frac{w^{n/(n-\alpha)}(B)}{|B|}\lesssim\inf_{x\in B} w(x)^{\frac n{n-\alpha}}\quad \text{for}\quad B\in \mathcal{B}_{\kappa,\text{loc}},
\end{equation}
and it characterizes the weak-type boundedness of $M_{\kappa, \alpha, \text{loc}}$ from $L^1(w)$ to $L^{\frac n{n-\alpha},\infty}(w)$.

In the mentioned papers the authors also proved that the conditions are independent of $\kappa$, so that we do not use this parameter in the notation of the weight class. It is clear that the usual $A_{p,q}$ weights are in $A_{p,q,\text{loc}}$ and it is immediate to check that $|x|^\alpha w(x)\in A_{p,q,\text{loc}}$ whenever $w\in A_{p,q}$ and $\alpha\in \R$. 

In \cite{HSV19} the local fractional integral $I_{\kappa,\alpha,\text{loc}}$ is defined as
\begin{equation}\label{deffracintloc}
I_{\kappa,\alpha,\text{loc}}f(x)=\int_{\{y: |y-x|<\kappa |x|\}}\frac {f(y)}{|x-y|^{n-\alpha}}\,dy.
\end{equation} 
It is proved that the boundedness from $L^p(w^p)$ to $L^q(w^q)$ of $I_{\kappa,\alpha,\text{loc}}$ is characterized also by \eqref{defapqloc}, and the weak-type for $p=1$ by \eqref{defa1loc}.

A local version of Calder\'on-Zygmund operators was also introduced in \cite{HSV19}. We do not need this general concept for our work, but only the following particular result. 

\begin{theorem}\cite[Corollary 5.2]{HSV19}\label{hsvcz}
Let $T$ be a Calder\'on-Zygmund operator in $\rn$ and $\kappa\in (0,1)$. Define 
\begin{equation}\label{localizedcz}
T_\kappa f(x)=T(f\chi_{B(x, \kappa |x|)})(x).
\end{equation}
Then $T_\kappa$ is bounded on $L^p(w)$ for $1<p<\infty$ and $w\in A_{p,\text{loc}}$, and of weak type (1,1) with respect to weights in $A_{1,\text{loc}}$.
\end{theorem}

\section{Necessary condition for singular integrals}\label{hiru}

A Calder\'on-Zygmund operator $T$ is an operator bounded on $L^2$ and associated with a standard kernel $K$ in the sense that if $f$ is a compactly supported $\mathcal C^\infty$ function, and $x\notin \text{supp}\,f$, $Tf(x)$ is given by
\begin{equation}\label{kernelrepr}
Tf(x)=\int_{\rn} K(x,y) f(y)\, dy.
\end{equation}
A standard kernel $K$ must satisfiy a size condition
\begin{equation}\label{kernelsize}
|K(x,y)|\lesssim \frac 1{|x-y|^n},
\end{equation}
and a regularity condition which we will not need to detail here (see \cite [Subsection 8.1.1]{Gr14}, for instance). 
Using the $L^2$ boundedness of $T$ it is shown in \cite [Proposition 8.1.9]{Gr14}  that if $f$ is bounded and compactly supported, and $x\notin \text{supp}\,f$, the representation \eqref{kernelrepr} holds as an absolutely convergent integral. We will assume that when we extend the operator to local Morrey spaces we keep such representation for bounded compactly supported functions.

We obtain the necessary condition for Riesz transforms in the general setting of Banach lattices. We recall that the Riesz transform $R_j$ ($1\le j\le n$) is defined (usually up to a constant factor) as the following principal value integral:
\begin{equation}
R_jf(x)=\text{p.v.}\int_{\rn} \frac {x_j-y_j}{|x-y|^{n+1}}\,f(y) dy. 
\end{equation}
The definition makes sense for smooth functions. It is a Calder\'on-Zygmund operator whose kernel is given by 
\begin{equation*}
K(x,y) = \frac {x_j-y_j}{|x-y|^{n+1}}.
\end{equation*}

\begin{theorem}\label{necrieszteo}
 Let $X$ be a Banach lattice of measurable functions in $\rn$ that contains the characteristic functions of balls. If the Riesz transforms are bounded on $X$, then 
\begin{equation}\label{necriesz}
\sup_B \frac{\|\chi _B\|_{X}\|M\chi _B\|_{X'}}{|B|} <\infty,
\end{equation}
where the supremum is taken over all balls in $\rn$.
\end{theorem}

We give first a lemma, which will be used in the proof of the theorem.

\begin{lemma}\label{normcomp}
Let $B$ and $B'$ balls with the same radius such that $c_{B'}=c_B+4r_B\mathbf e_j$, where $\mathbf e_j$ is the unit vector in the direction of the $j$-th component. If $R_j$ is bounded on $X$, then $\|\chi_B\|_X\sim \|\chi_{B'}\|_X$.  
\end{lemma}

\begin{proof}
 Let $x\in B$ and $y\in B'$. We have $y_j>x_j$ and $y_j-x_j \sim |x-y| \sim r_B$. Then
 \begin{equation*}
|R_j\chi_{B'}(x)|= \int_{B'} \frac {y_j-x_j}{|x-y|^{n+1}} dy\sim 1. 
\end{equation*}
Hence $\chi_B (x) \lesssim |R_j\chi_{B'}(x)|$ and taking norms in $X$ and using the boundedness of $R_j$ we deduce that  $\|\chi_B\|_X\lesssim \|\chi_{B'}\|_X$. Changing the roles of $B$ and $B'$ we obtain the reverse inequality and the lemma is proved.
\end{proof}

\begin{remark}
 In the case of the Morrey spaces, the norm of $\chi_B$ in $L\mathcal M^{p}(\varphi,w)$ and in  $WL\mathcal M^{p}(\varphi,w)$ is the same. It is enough to assume that $R_j$ is bounded from $L\mathcal M^{p}(\varphi,w)$ to  $WL\mathcal M^{p}(\varphi,w)$ to get the conclusion of the lemma.
\end{remark}

\begin{proof}[Proof of Theorem \ref{necrieszteo}]
Let $B$ be a ball and $Q$ the cube with sides parallel to the coordinate axes and the same center and radius as $B$. Let $2Q$ be the cube with double radius. We can decompose $\rn\setminus 2Q$ into $2n$ regions, $\{P_j^+, P_j^-: j=1,\dots,n\}$ in such a way that for $x\in B$ and $y\in P_j^+$ (resp. $P_j^-$), we have
\begin{equation}\label{regions}
0<y_j-x_j\sim |x-y| \quad (\text{resp. }0<x_j-y_j\sim |x-y|).
\end{equation}
(See Figure \ref{figure1} for $n=2$.)

\begin{center}
\begin{figure}
\begin{tikzpicture}[scale=1]
\draw[thick, black] circle [radius=1];
\draw[dashed, black] (1,1) -- (-1,1) -- (-1,-1) -- (1,-1) -- (1,1);
\draw[thick, black] (2,2) -- (-2,2) -- (-2,-2) -- (2,-2) -- (2,2);
\draw[thick, black] (2,2) -- (4,4);
\draw[thick, black] (-2,2) -- (-4,4);
\draw[thick, black] (2,-2) -- (4,-4);
\draw[thick, black] (-2,-2) -- (-4,-4);
\draw (0.3, 0) node {$B$};
\draw (1.2, 0.3) node {$Q$};
\draw (1.5, 1.5) node {$2Q$};
\draw (3, 0) node {$P_1^+$};
\draw (-3, 0) node {$P_1^-$};
\draw (0,3) node {$P_2^+$};
\draw (0,-3) node {$P_2^-$};
\end{tikzpicture}
\caption{The decomposition used in the proof of Theorem \ref{necrieszteo} for $n=2$.}  
\label{figure1}
\end{figure}
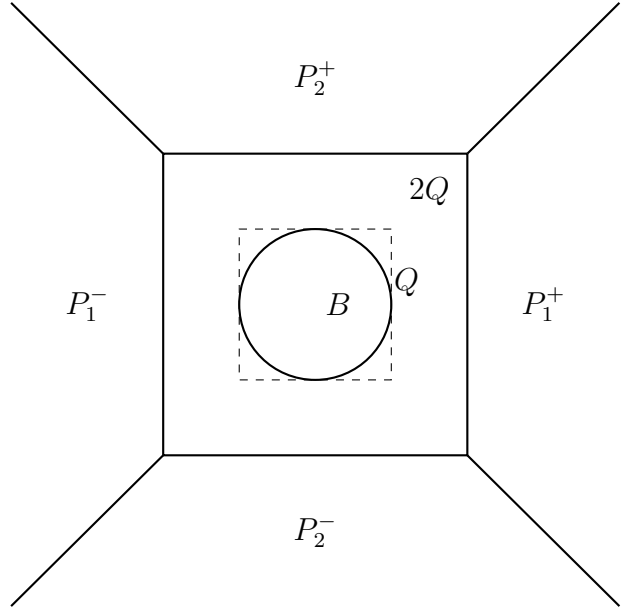
\end{center}

Let $f$ be a bounded compactly supported function. Then for $x\in B$ we have 
\begin{equation*}
|R_j(|f|\chi_{P_j^+})(x)|= \int_{P_j^+} \frac {y_j-x_j}{|x-y|^{n+1}} |f(y)|\,dy \sim  \int_{P_j^+} \frac {|f(y)|}{|c_B-y|^{n}} dy,
\end{equation*}
and similarly for $P_j^-$. Therefore,
\begin{equation*}
\int_{P_j^+\cup P_j^-} |f(y)|\frac {M\chi_B(y)}{|B|} dy \chi_B(x) \lesssim |R_j(|f|\chi_{P_j^+\cup P_j^-})(x)|.
\end{equation*}
Taking norms in $X$ and using the assumed boundedness of $R_j$ we deduce that
\begin{equation}\label{est31}
\frac 1 {|B|}\int_{P_j^+\cup P_j^-} |f(y)|{M\chi_B(y)} dy\  \|\chi_B\|_X \lesssim \|f\|_X.
\end{equation}

Let $B'$ be the ball with the same radius as $B$ and center $c_{B'}=c_B+4r_B\mathbf e_1$. For $x\in B'$ and $y\in 2Q$ we have $0<x_1-y_1 \sim r_B\sim |x-y|$. Then
\begin{equation*}
|R_1(|f|\chi_{2Q})(x)|= \int_{2Q} \frac {x_1-y_1}{|x-y|^{n+1}} |f(y)|dy \sim  \int_{2Q} \frac {|f(y)|}{r_B^{n}} dy.
\end{equation*}
Moreover, $M\chi_B(y) \sim 1$ for $y\in 2Q$. Hence
\begin{equation*}
\int_{2Q} |f(y)|\frac {M\chi_B(y)}{|B|} dy \ \chi_{B'}(x) \lesssim |R_1(|f|\chi_{2Q})(x)|.
\end{equation*}
Taking norms in $X$, using the boundedness of $R_1$, and Lemma \ref{normcomp} we obtain
\begin{equation*}
\frac 1 {|B|}\int_{2Q} |f(y)|{M\chi_B(y)} dy\  \|\chi_B\|_X \lesssim \|f\|_X.
\end{equation*}
Adding this estimate to the sum for $j=1,\dots,n$ of the estimates \eqref{est31} we get
\begin{equation}\label{eq37}
\frac 1 {|B|}\int_{\rn} |f(y)|{M\chi_B(y)} dy\  \|\chi_B\|_X \lesssim \|f\|_X.
\end{equation}

We obtained \eqref{eq37} for bounded functions with compact support. Given a general $f$ in $X$, for each $N>1$ we define 
\begin{equation*}
f_N(x) = \min \{|f(x)|,N\} \chi_{|x|\le N}(x).
\end{equation*}
Since $X$ is a Banach lattice, $f_N\in X$ and $\|f_N\|_X\le \|f\|_X$. We can apply \eqref{eq37} to  $f_N$ to obtain 
\begin{equation*}
\frac 1 {|B|}\int_{\rn} |f_N(y)|{M\chi_B(y)} dy\  \|\chi_B\|_X \lesssim \|f_N\|_X\le \|f\|_X.
\end{equation*}
Using the monotone convergence theorem we get \eqref{eq37} for every $f\in X$. 
Taking the supremum over the functions $f$ such that $\|f\|_X\le 1$ we obtain \eqref{necriesz}.
\end{proof}

\section{Calder\'on-Zygmund operators on weighted local Morrey spaces}\label{lau}

Calder\'on-Zygmund operators are first defined on smooth functions and can be extended by density to other spaces like weighted $L^p$ spaces, for instance. Nevertheless, due to the lack of density of smooth functions on the weighted local Morrey spaces used in the results of this section, we need to give an appropriate definition. To this end, we proceed as follows. Let $T$ be a Calder\'on-Zygmund operator associated to the kernel $K$ and $f\in L\mathcal M^{p}(\varphi,w)$ with $w$ satisfying \eqref{sicond}. Since \eqref{sicond} is stronger than \eqref{apmdef}, we know that $w\in A_{p,\textrm{loc}}$ (see \cite[Lemma 6.2]{DR21}). Let $B$ be a ball such that $|c_B|=6r_B$. Decompose $f$ as $f=f_1+f_2$, where $f_1=f\chi_{2B}$.   We have that $f_1$ is in $L^1$, hence $Tf_1$ is well defined. The integrability of $f_1$ is a consequence of the inequality
\begin{equation*}
\int_{2B} |f| \le \left(\int_{2B} |f|^p w\right)^{1/p} w^{1-p'}(2B)^{1/p'},
\end{equation*}
where the last term is finite because $w\in A_{p,\text{loc}}$. (For $p=1$ the last term is replaced by $\sup_{x\in 2B} w(x)^{-1}$.)

On the other hand, for $x\in B$, we define
\begin{equation}\label{deftf2}
Tf_2(x)=\int_{(2B)^c} K(x,y)f(y) dy.
\end{equation}
This integral is well defined because for $x\in B$ we have
\begin{equation*}
|K(x,y) \chi_{(2B)^c}(y)|\lesssim \frac 1{|B|} M\chi_B(y),
\end{equation*}
using the size condition \eqref{kernelsize} of the kernel $K$, and $M\chi_B\in  L\mathcal M^{p}(\varphi,w)'$ by \eqref{sicond}.  Moreover, we have the following pointwise estimate:
\begin{equation}\label{pointest}
\left|\int_{(2B)^c} K(x,y)f(y) dy\right|\lesssim  \frac 1{|B|} \|M\chi_B\|_{L\mathcal M^{p}(\varphi,w)'}\|f\|_{L\mathcal M^{p}(\varphi,w)}.
\end{equation}
Thus we define $Tf$ on $B$ as $Tf_1+Tf_2$. We need to check that if $B'$ is another ball such that $|c_{B'}|=6r_{B'}$ and $B\cap B'\ne \emptyset$, both definitions of $Tf$ coincide almost everywhere on $B\cap B'$, that is,
\begin{equation*}
T(f\chi_{2B})(x)+\int_{(2B)^c} K(x,y)f(y) dy = T(f\chi_{2B'})(x)+\int_{(2B')^c} K(x,y)f(y) dy,
\end{equation*}
for almost every $x\in B\cap B'$. This can be written as 
\begin{equation*}
\aligned
T(f(\chi_{2B}-\chi_{2B'}))(x) & = \int_{2B\setminus 2B'}K(x,y)f(y) dy- \int_{2B'\setminus 2B}K(x,y)f(y) dy\\
& = \int_{\rn}K(x,y)\, [f(\chi_{2B}-\chi_{2B'})](y) dy.
\endaligned
\end{equation*}    
The equality holds because $f(\chi_{2B}-\chi_{2B'})$ is in $L^1$ and $x\in B\cap B'$ is outside the support of $f$. Indeed, the  representation formula \eqref{kernelrepr} can be used for $L^1$ functions as proved in \cite[Proposition 8.2.2]{Gr14}. 

A priori estimates with Morrey norms are not enough to extend singular integral operators (and others) to the whole Morrey space. This question, disregarded in a number of papers, was addressed already in \cite{A96}, where duality was used to define the Calder\'on-Zygmund operator on the corresponding Morrey space. Since then different ways had been considered as can be seen, for instance, in \cite[Section 10.4, vol. I]{SFH20}. Our definition can be viewed as the natural extension to the weighted local Morrey space of the definition of the operator in $L^1$ in the sense that for functions in $L^1\cap  L\mathcal M^{p}(\varphi,w)$ they coincide. In \cite{H20}, for instance, a similar definition appears in a quite general context.

We notice that for $x\in B$ we have $2B\subset B(x,3/5|x|)$. Indeed, if $y\in 2B$, 
\begin{equation*}
|y-x|\le |y-c_B|+|c_B-x|\le 3r_B\le \frac 35 |x|,
\end{equation*}
because $|c_B|= 6r_B$. Then $f\chi_{2B}=(f\chi_{2B})\chi_{B(x,3/5|x|)}$, hence
\begin{equation}\label{deftf1}
T(f\chi_{2B})= T_{3/5}(f\chi_{2B})
\end{equation}
in the sense of \eqref{localizedcz}. By using \cite[Corollary 5.2]{HSV19} we know that $T_{3/5}$ is bounded on $L^p(w)$ for $w\in A_{p,\text{loc}}$ for $1<p<\infty$, and is of weak-type $(1,1)$ with respect to $w$.

\begin{theorem}\label{teocz}
  Let $1< p<\infty$. Let $T$ be a Calder\'on-Zygmund operator and let $w$ be a weight satisfying \eqref{sicond}.
 Then the operator $T$ defined as indicated is bounded on $L\mathcal M^{p}(\varphi,w)$. 
 
 Let $p=1$. Then $T$ is of weak type $(1,1)$, that is, it is bounded from $L\mathcal M^{1}(\varphi,w)$ to $WL\mathcal M^{1}(\varphi,w)$.
\end{theorem}

\begin{proof}
Let $B$ be a ball such that $|c_B|=6r_B$. Decompose $f$ as $f_1+f_2$ as before. On the one hand, using \eqref{pointest} we have
\begin{equation*}
\aligned
\left(\frac 1{\varphi(B)} \int_B |Tf_2|^pw\right)^{\frac 1p} & \lesssim \left(\frac {w(B)}{\varphi(B)}\right)^{\frac 1p} \frac 1{|B|} \|M\chi_B\|_{L\mathcal M^{p}(\varphi,w)'}\|f\|_{L\mathcal M^{p}(\varphi,w)} \\
& \lesssim \|\chi_B\|_{L\mathcal M^{p}(\varphi,w)}  \frac 1{|B|} \|M\chi_B\|_{L\mathcal M^{p}(\varphi,w)'}\|f\|_{L\mathcal M^{p}(\varphi,w)}.
\endaligned
\end{equation*}

On the other hand, taking into account \eqref{deftf1}  and that $T_{3/5}$ is bounded on $L^p(w)$ for $w\in A_{p,\text{loc}}$, we have
\begin{equation*}
\aligned
\frac 1{\varphi(B)} \int_B |Tf_1|^pw & = \frac 1{\varphi(B)} \int_B |T_{3/5}f_1|^pw \\
& \lesssim   \frac 1{\varphi(B)} \int_{2B} |f|^pw \lesssim  \|f\|_{L\mathcal M^{p}(\varphi,w)}^p.
\endaligned
\end{equation*}
The theorem is proved for $1<p<\infty$.

In the case $p=1$ the part corresponding to $Tf_2$ is still valid, and for $Tf_1$ we replace the $L^p(w)$-boundedness by the weak-type (1,1) estimate for $T_{3/5}$.
\end{proof}

Combining this result with the necessary condition of the previous section we obtain the following corollary. 

\begin{corollary}
The Riesz transforms are bounded on $L\mathcal M^{p}(\varphi,w)$ for $1<p<\infty$ if and only if $w$ satisfies \eqref{sicond}. The Riesz transforms are bounded from $L\mathcal M^{1}(\varphi,w)$ to $WL\mathcal M^{1}(\varphi,w)$ if and only if \eqref{sicond} holds for $p=1$.
\end{corollary}

\section{Necessary conditions for fractional operators}\label{bost}

We recall that the fractional maximal operator $M_\alpha$ is defined for locally integrable functions in $\rn$ as in \eqref{defmaxloc}, but without restrictions on the balls, that is, the supremum is taken over all balls containing $x$. The fractional integral $I_\alpha$ is defined as in \eqref{deffracintloc} with the integral extended to all of $\rn$. We also recall the pointwise bound $M_\alpha f(x) \le I_\alpha f(x)$.

\begin{theorem}\label{teonecfrac} 
Let $X$ be a Banach lattice and $Y$ a normed space.

(a) If the fractional maximal operator $M_\alpha$ is bounded from $X$ to $Y$, then 
 \begin{equation}\label{necfracmax}
\sup_B \frac{\|\chi _B\|_{Y}\|\chi _B\|_{X'}}{|B|^{1-\alpha/n}} <\infty.
\end{equation}
(b) If the fractional integral operator $I_\alpha$ is bounded from $X$ to $Y$, then 
 \begin{equation}\label{necfracint}
\sup_B \frac{\|\chi _B\|_{Y}\|(M\chi _B)^{1-\alpha/n}\|_{X'}}{|B|^{1-\alpha/n}} <\infty.
\end{equation}
\end{theorem}

\begin{proof}
(a) Let $B$ be a ball. By definition of $M_\alpha$ we have
\begin{equation*}
\frac 1{|B|^{1-\alpha/n}}\int_B |f| \le M_\alpha f(x)\quad \text{for all } x\in B,
\end{equation*}
that is,
\begin{equation*}
\left(\frac 1{|B|^{1-\alpha/n}}\int_B |f|\right) \chi_B(x) \le M_\alpha f(x).
\end{equation*}
Taking the norm in $Y$ and assuming that $M_\alpha$ is bounded from $X$ to $Y$,
\begin{equation*}
\left(\frac 1{|B|^{1-\alpha/n}}\int \chi_B |f|\right)\|\chi_B\|_Y \le \|M_\alpha f\|_Y\lesssim \|f\|_X.
\end{equation*}
Taking the supremum over $f$ such that $\|f\|_X\le 1$ we obtain \eqref{necfracmax}.

(b) Let $B$ be a ball and $x\in B$. Then
$|x-y|\lesssim r_B + |c_B-y|$. Hence,
\begin{equation*}
M\chi_B(y) \sim \frac {|B|}{(r_B + |c_B-y|)^n}\lesssim  \frac {|B|}{|x-y|^n}.
\end{equation*}
Therefore, for $x\in B$,
\begin{equation*}
I_\alpha (|f|)(x) = \int \frac {|f(y)|}{|x-y|^{n-\alpha}}\,dy\gtrsim \int |f(y)| \left(\frac {M\chi_B(y)}{|B|}\right)^{1-\frac \alpha n}\,dy.
\end{equation*}
We deduce that
\begin{equation*}
\left[\frac {1}{|B|^{1-\frac \alpha n}} \int |f| (M\chi_B)^{1-\frac \alpha n}\right] \chi_B(x) \lesssim I_\alpha (|f|)(x),
\end{equation*}
and proceeding as in part (a) we obtain \eqref{necfracint}.
\end{proof}

We obtain in the following theorem a restriction on the exponents of the local Morrey spaces for the boundedness of the fractional operators. It is enough to state it for $M_\alpha$ because of the pointwise bound between $M_\alpha$ and $I_\alpha$.

\begin{theorem}\label{teonecrange}
If the fractional maximal operator $M_\alpha$ is bounded from $\mathcal {LM}^{p}(\varphi,u)$ to $W\mathcal {LM}^{q}(\psi,v)$, then 
 \begin{equation}\label{necrange}
\frac 1p -\frac 1q \le \frac \alpha n.
\end{equation}
\end{theorem}

\begin{proof}
Let $x\ne 0$ be a Lebesgue point of both $u$ and $v$, such that $v(x)\ne 0$. Let $B_r=B(x,r)$ with $4r\le |x|$. Let $\widehat B=B(0,2|x|)$. Then
\begin{equation*}
\|\chi_{B_r}\|_{\mathcal {LM}^{p}(\varphi,u)}\sim  \left(\frac {u(B_r)}{\varphi(\widetilde B_r)}\right)^{1/p}\sim  \left(\frac {u(B_r)}{\varphi(\widehat B)}\right)^{1/p},
\end{equation*}
and
\begin{equation*}
\|\chi_{B_r}\|_{W\mathcal {LM}^{q}(\psi,v)}=  \|\chi_{B_r}\|_{\mathcal {LM}^{q}(\psi,v)}\sim  \left(\frac {v(B_r)}{\psi(\widehat B)}\right)^{1/q}.
\end{equation*}
Using that $r^\alpha \chi_{B_r}(y) \le M_\alpha \chi_{B_r}(y)$, from the boundedness assumption we deduce
\begin{equation*}
r^\alpha  \left(\frac {v(B_r)}{\psi(\widehat B)}\right)^{1/q} \lesssim \left(\frac {u(B_r)}{\varphi(\widehat B)}\right)^{1/p},
\end{equation*}
which can be written as
\begin{equation}\label{seisei}
r^{\alpha+n/q-n/p}  \left(\frac {v(B_r)}{r^n}\right)^{1/q} \lesssim \left(\frac {u(B_r)}{r^n}\right)^{1/p}\frac{\psi(\widehat B)^{1/q}}{\varphi(\widehat B)^{1/p}}.
\end{equation}
Since $x$ is a Lebesgue point of $u$ and $v$, we have 
\begin{equation*}
\lim_{r\to 0} \frac {v(B_r)}{r^n}=v(x) \quad \text{and} \quad \lim_{r\to 0} \frac {u(B_r)}{r^n}=u(x).
\end{equation*}
Hence, letting $r$ tend to zero in \eqref{seisei}, we see that the inequality can only hold if $\alpha+n/q-n/p\ge 0$.
\end{proof}

This theorem extends to a general weighted setting a result which is known in the unweighted setting with $\varphi(B)=\psi(B)=r_B^\lambda$. A counterexample is in \cite{KFS17}; see also \cite[Example 158]{SFH20}. This is in contrast with the case of global Morrey spaces, where Adams extended the unweighted result to the range $1/p-1/q=\alpha/(n-\lambda)$. For weighted inequalities in the Adams range on global Morrey spaces see \cite{IKS11} and  \cite{NST18}. Let us mention that it is possible to obtain an Adams-type result on (unweighted) local Morrey spaces if we restrict the action of the operator to radial functions (see \cite{KFS17} and \cite{KFS17b}).

\section{Sufficient conditions for fractional operators}\label{sei}

Before characterizing the weights for $M_{\alpha}$ we prove a lemma, similar to Lemma 6.2  in \cite{DR21}. We use the notation $\varphi^p$ to denote the function such that $\varphi^p(B)=\varphi(B)^p$ as mentioned in the introduction.

\begin{lemma}\label{lema61}
Let $1\le p<n/\alpha$ and $1/p-1/q=\alpha/n$.  Then $w\in A_{p,q,\text{loc}}$ as defined in \eqref{defapqloc} if and only if $w$ satisfies \eqref{suffracmax1} restricted to the balls $B$ such that $r_B\le \kappa |c_B|$ for some fixed $\kappa\in (0,1)$. 
\end{lemma}

\begin{proof}
Let $B$ be a ball such that $r_B\le \kappa |c_B|$. We have from \eqref{locchar} that 
\begin{equation}\label{charq} 
\|\chi_B\|_{\mathcal {LM}^{q}(\varphi^q,w^q)}\sim \left(\frac{w^q(B)}{\varphi^q(\widetilde B)}\right)^{1/q}=\frac{w^q(B)^{1/q}}{\varphi (\widetilde B)}.
\end{equation}

Let $p>1$. We claim that 
\begin{equation}\label{normprime}
\|\chi_B\|_{\mathcal {LM}^{p}(\varphi^p,w^p)'}\sim w^{-p'}(B)^{1/p'}\varphi(\widetilde{B}).
\end{equation}

On the one hand we have 
\begin{equation*}
\|\chi_B\|_{\mathcal {LM}^{p}(\varphi^p,w^p)'}\ge \frac{1}{\|w^{-p'}\chi_B\|_{\mathcal {LM}^{p}(\varphi^p,w^p)}}\int_B w^{-p'} \sim
 w^{-p'}(B)^{1/p'} \varphi(\widetilde{B}),
\end{equation*}
where we used \eqref{locloc} for the norm of $w^{-p'}\chi_B$.

On the other hand, 
\begin{equation*}
\int_B |f|\le \left(\int_B |f|^p w^p\right)^{1/p} w^{-p'}(B)^{1/p'} 
\le  \varphi(\widetilde{B}) \|f\|_{\mathcal {LM}^{p}(\varphi^p,w^p)}
w^{-p'}(B)^{1/p'}.
\end{equation*}
Taking the supremum over the functions $f$ with $\|f\|_{\mathcal {LM}^{p}(\varphi^p,w^p)}\le 1$ we prove the claim \eqref{normprime}.

Therefore,
\begin{equation*}
\frac{\|\chi_B\|_{\mathcal {LM}^{q}(\varphi^q,w^q)} \|\chi_B\|_{\mathcal {LM}^{p}(\varphi^p,w^p)'}}{|B|^{1-\alpha/n}} \sim 
\frac{w^q(B)^{1/q} w^{-p'}(B)^{1/p'}} {|B|^{1-\alpha/n}},
\end{equation*}
and this proves the lemma for $p>1$.

Let $p=1$.  We have \eqref{charq} with $q=n/(n-\alpha)$. Let $w\in A_{1,n/(n-\alpha),\text{loc}}$. Then 
\begin{equation*}
\aligned
\int_B |f| & \lesssim \int_B |f|w\  \left(\frac{|B|}{w^{n/(n-\alpha)}(B)}\right)^{1-\alpha/n}\\
& \le \varphi(\widetilde{B})\left(\frac{|B|}{w^{n/(n-\alpha)}(B)}\right)^{1-\alpha/n} \|f\|_{\mathcal {LM}^{1}(\varphi,w)}.
\endaligned
 \end{equation*}
 From here we obtain
 \begin{equation*}
\|\chi_B\|_{\mathcal {LM}^{1}(\varphi,w)'}\lesssim \varphi(\widetilde{B})\left(\frac{|B|}{w^{n/(n-\alpha)}(B)}\right)^{1-\alpha/n},
 \end{equation*}
 which together with \eqref{charq} implies that $w$ satisfies \eqref{suffracmax1} for the selected balls.
 
 Let $E$ be a subset of $B$ of positive measure. Then
 \begin{equation*}
\|\chi_B\|_{\mathcal {LM}^{1}(\varphi,w)'}\ge \frac{1}{\|\chi_E\|_{\mathcal {LM}^{1}(\varphi,w)}}\int_B \chi_E \sim \frac{\varphi(\widetilde{B})|E|}{w(E)}. 
\end{equation*}
Using \eqref{charq} and \eqref{suffracmax1} we obtain
\begin{equation}\label{fora1}
\left(\frac {w^{n/(n-\alpha)}(B)}{|B|}\right)^{1-\alpha/n}\lesssim \frac{w(E)}{|E|}.
\end{equation}
 With $E=\{x\in B: w(x)\le \inf_B w+\epsilon\}$ we obtain
\begin{equation*}
\left(\frac {w^{n/(n-\alpha)}(B)}{|B|}\right)^{1-\alpha/n}\lesssim \inf_B w+\epsilon.
\end{equation*}
Since this holds for any $\epsilon>0$ we get the $A_{1,n/(n-\alpha),\text{loc}}$ condition. 
\end{proof}

In the following theorem we obtain the characterization of weighted inequalities for $M_\alpha$ and $I_\alpha$. 

\begin{theorem}\label{teosuffrac1}
Let $0<\alpha<n$, $1<p<n/\alpha$, and $1/p-1/q=\alpha/n$. 

(a) The fractional maximal operator $M_\alpha$ is bounded from $L\mathcal M^{p}(\varphi^p,w^p)$ to $L\mathcal M^{q}(\varphi^q,w^q)$ if and only if \eqref{suffracmax1} holds.

(b) The fractional integral operator $I_\alpha$ is bounded from $L\mathcal M^{p}(\varphi^p,w^p)$ to $L\mathcal M^{q}(\varphi^q,w^q)$ if and only if \eqref{suffracint1} holds.
 
(c) For $p=1$ and $q= n/(n-\alpha)$ both results hold with $M_\alpha$ and $I_\alpha$ bounded from $L\mathcal M^{1}(\varphi,w)$ to $WL\mathcal M^{\frac n{n-\alpha}}(\varphi^{\frac n{n-\alpha}},w^{\frac n{n-\alpha}})$.
\end{theorem}

\begin{proof} 
The necessity of all cases has been proved in the previous section. Let us prove the sufficiency.

(a) Let $B$ with $|c_B|=6r_B$. Given a nonnegative function $f$ we decompose it as $f=f_1+f_2$, where $f_1=f\chi_{2B}$. Using the subadditivity of $M_\alpha$ we have
\begin{equation*}
M_\alpha f(y)\le M_\alpha f_1(y)+M_\alpha f_2(y).
\end{equation*}
 
Since $f_1$ is supported in $2B$, to compute $M_\alpha f_1(y)$ for $y\in B$ we only need to consider balls $B(y,\rho)$ with $\rho<3r_B$. Moreover, for $r_B<\rho<3r_B$ we can replace $B(y,\rho)$ with $2B$. Therefore,
\begin{equation*}
\aligned
M_\alpha f_1(y)&\sim \sup_{\rho \le r_B} \frac 1{|B(y,\rho)|^{1-\frac \alpha n}}\int_{B(y,\rho)}|f|+\frac 1{|2B|^{1-\frac \alpha n}}\int_{2B}|f|\\
&\sim M_{1/5,\alpha,\textrm{loc}}f_1(y)+\frac 1{|2B|^{1-\frac \alpha n}}\int_{2B}|f|.
\endaligned
\end{equation*}

On the other hand, since $f_2$ is supported on $(2B)^c$, to compute $M_\alpha f_2(y)$ for $y\in B$ one  only needs to consider balls of radius greater than $r_B$, and as a consequence, $M_\alpha f_2$ is almost constant on $B$. Then \begin{equation*}
M_\alpha f_2(y)\sim \sup_{B' \supset B}\frac 1{|B'|^{1-\frac \alpha n}}\int_{B'}|f_2|\quad \text{ for } y\in B.
\end{equation*} 

Altogether, we have for all $y\in B$ that
\begin{equation}\label{equivm}
M_\alpha f(y)\sim M_{1/5,\alpha,\textrm{loc}}(f \chi_{2B})(y) + \sup_{B' \supset B} \frac 1{|B'|^{1-\frac \alpha n}}\int_{B'}|f_2|.
\end{equation} 

For the last term we have 
\begin{equation*}
\aligned
 \frac 1{|B'|^{1-\frac \alpha n}}\int_{B'}|f| &\le  \frac 1{|B'|^{1-\frac \alpha n}}\|f\|_{L\mathcal M^{p}(\varphi^p,w^p)}\|\chi_{B'}\|_{L\mathcal M^{p}(\varphi^p,w^p)'}\\
 & \lesssim \frac{\|f\|_{L\mathcal M^{p}(\varphi^p,w^p)}}{\|\chi_{B'}\|_{L\mathcal M^{q}(\varphi^q,w^q)}}\le 
 \frac{\|f\|_{L\mathcal M^{p}(\varphi^p,w^p)}}{\|\chi_{B}\|_{L\mathcal M^{q}(\varphi^q,w^q)}},
 \endaligned
\end{equation*}
where we used successively \eqref{holder}, \eqref{suffracmax1} for $B'$, and $B\subset B'$.

Now we have
\begin{equation*}
\aligned
\ & \left(\frac 1{\varphi(\widetilde B)^q}\int_B (M_\alpha f)^q w \right)^{1/q} \\
&\qquad  \lesssim
 \left(\frac 1{\varphi(\widetilde B)^q}\int_B \left(M_{1/5,\alpha,\textrm{loc}}(f \chi_{2B})+\frac{\|f\|_{L\mathcal M^{p}(\varphi^p,w^p)}}{\|\chi_{B}\|_{L\mathcal M^{q}(\varphi^q,w^q)}}\right)^q w^q\right)^{1/q}\\ 
&\qquad  \lesssim \left(\frac 1{\varphi(\widetilde B)^q}\int_{2B} |f|^p w^p\right)^{1/p} + \|f\|_{L\mathcal M^{p}(\varphi^p,w^p)},
 \endaligned
\end{equation*}
where we used that $w\in A_{p,q,\text{loc}}$, which is implied by \eqref{suffracmax1} according to Lemma \ref{lema61}. Taking into account that $\widetilde{2B}\subset 2 \widetilde{B}$, we obtain the required bound using the doubling property of $\varphi$.

(b)  Let $f$ be nonnegative. We take a ball $B$ as in part (a) and decompose $f$ as $f_1+f_2$. From the linearity of $I_\alpha$ we have 
\begin{equation*}
I_\alpha f= I_\alpha f_1+I_\alpha f_2.
\end{equation*}
For $x\in B$ we have $I_\alpha f_1(x)=I_{3/5,\alpha,\text{loc}} f_1(x)$, where the local fractional operator has been defined in \eqref{deffracintloc}. Condition \eqref{suffracint1} is stronger than \eqref{suffracmax1}, hence $w\in A_{p,q,\text{loc}}$.
Using the weighted boundedness of $I_{3/5,\alpha,\text{loc}}$ from $L^p(w^p)$ to $L^q(w^q)$ for $w\in A_{p,q,\text{loc}}$, this term works as the first term in the right-hand side of \eqref{equivm}. 

On the other hand, for $x\in B$ and $y\in \rn\setminus 2B$ we have $M\chi_B(y) \sim |B| |x-y|^{-n}$. Hence,
\begin{equation*}
I_\alpha f_2(x) =\int_{\rn\setminus 2B}\frac {f(y)}{|x-y|^{n-\alpha}}\lesssim \int_{\rn} f(y) \left(\frac {M\chi_B(y)}{|B|}\right)^{1-\frac \alpha n}\,dy.
\end{equation*}
Using \eqref{holder} we obtain 
\begin{equation*}
\aligned
& \left(\frac 1{\varphi(\widetilde B)^q}\int_B (I_\alpha f_2)^q w^q \right)^{1/q}\\
& \qquad\le  \frac {\|f\|_{L\mathcal M^{p}(\varphi^p,w^p)} \|(M\chi_B)^{1-\alpha/n}\|_{L\mathcal M^{p}(\varphi^p,w^p)'}}{|B|^{1-\alpha/n}}\left(\frac 1{\varphi(\widetilde B)^q} \int_Bw^q\right)^{1/q}
 \endaligned
\end{equation*}
and \eqref{suffracint1} suffices to conclude.

(c) For $p=1$ we proceed as in (a) and (b), but when dealing with the local operators acting on $f_1$ the weighted inequalities are of weak-type. The rest of the proof is the same. 
\end{proof}

\begin{remark}\label{rem65}
For weights $w$ in $A_{p,q,\text{loc}}$ one only needs to verify \eqref{suffracmax1} and \eqref{suffracint1} for balls centered at the origin. Balls $B$ such that $|c_B|\le 6r_B$, can be replaced by $\widetilde B$, because $B\subset \widetilde B$ and $|B|\sim |\widetilde B|$. For balls with $|c_B|> 6r_B$, $A_{p,q,\text{loc}}$ is enough for $M_\alpha$ due to Lemma \ref{lema61}, but also for $I_\alpha$, as a consequence of the proof of Theorem \ref{teosuffrac1}, where to obtain the bound for $I_\alpha f_1$ we only use that $w\in A_{p,q,\text{loc}}$. 
\end{remark}

The proof of the theorem can be adapted to the case $1/p-1/q<\alpha/n$ with appropriate weights.

\begin{theorem}\label{teogamma}
Let $0< \alpha<n$, $0\le \gamma < \alpha$, $1\le p<n/\gamma$ and $1/p-1/q= \gamma/n$. Assume that $w\in A_{p,q,\text{loc}}$. Set $v_\gamma(x)= |x|^{\gamma-\alpha} w(x)$.

Let $p>1$. The fractional maximal operator $M_\alpha$ is bounded from $L\mathcal M^{p}(\varphi^p,w^p)$ to $L\mathcal M^{q}(\varphi^q,v_\gamma^q)$ if 
\begin{equation}\label{frmaxext}
\sup_B \frac{\|\chi _B\|_{L\mathcal M^{q}(\varphi^q,v_\gamma^q)}\|\chi _B\|_{L\mathcal M^{p}(\varphi^p,w^p)'}}{|B|^{1-\alpha/n}} <\infty
\end{equation}
for balls centered at the origin.

The fractional integral operator $I_\alpha$ is bounded from $L\mathcal M^{p}(\varphi^p,w^p)$ to $L\mathcal M^{q}(\varphi^q,v_\gamma^q)$ if
\begin{equation}\label{frintext}
\sup_B \frac{\|\chi _B\|_{L\mathcal M^{q}(\varphi^q,v_\gamma^q)}\|(M\chi _B)^{1-\alpha/n}\|_{L\mathcal M^{p}(\varphi^p,w^p)'}}{|B|^{1-\alpha/n}} <\infty.
\end{equation}
for balls centered at the origin.

 For $p=1$ and $q= n/(n-\gamma)$ the results hold for the boundedness from $L\mathcal M^{1}(\varphi,w)$ to $WL\mathcal M^{\frac n{n-\gamma}}(\varphi^{\frac n{n-\gamma}},v_\gamma^{\frac n{n-\gamma}})$.
\end{theorem}

\begin{proof}
First we proof the following pointwise inequalities:
\begin{equation}\label{pbm}
M_{\kappa,\alpha,\text{loc}} f(x), I_{\kappa,\alpha,\text{loc}}(|f|)(x)\lesssim |x|^{\alpha -\gamma} M_{\kappa,\gamma,\text{loc}} f(x).
\end{equation}
Let $B$ a ball containing $x$ and such that $r_B<\kappa |c_B|$. Then 
\begin{equation*}
\frac 1{|B|^{1-\alpha/n}}\int_{B}|f(y)|\,dy\le \frac {r_B^{\alpha-\gamma}}{|B|^{1-\gamma/n}}\int_{B}|f(y)|\,dy
\lesssim
\frac {|x|^{\alpha-\gamma}}{|B|^{1-\gamma/n}}\int_{B}|f(y)\,dy,
\end{equation*}
from which \eqref{pbm} follows for $M_{\kappa,\alpha,\text{loc}}$.

On the other hand,
\begin{equation*}
\aligned
\int_{\{y: |y-x|<\kappa |x|\}} & \frac{|f(y)|}{|x-y|^{n-\alpha}}\,dy = \sum_{j=0}^\infty \int_{\{y: 2^{-j-1}\kappa |x|\le |y-x|<2^{-j}\kappa |x|\}}\frac{|f(y)|}{|x-y|^{n-\alpha}}\,dy \\
&\lesssim \sum_{j=0}^\infty (2^{-j}\kappa |x|)^{\alpha-\gamma} \frac 1{(2^{-j}\kappa |x|)^{n-\gamma}}\int_{\{y: |y-x|<2^{-j}\kappa |x|\}}|f(y)|\,dy\\
& \lesssim |x|^{\alpha-\gamma} M_{\kappa,\gamma,\text{loc}} f(x).
\endaligned
\end{equation*}

We proceed as in the proof of the previous theorem. The action of the operators on $f_2$ is treated in the same way as before. When working with $f_1$, using \eqref{pbm} the needed estimate is shifted to $M_{\kappa,\gamma,\text{loc}}$. 
\end{proof}

From Theorem \ref{teonecfrac} we know that \eqref{frmaxext} and \eqref{frintext} for all balls in $\rn$ are necessary  for the boundedness of the operators. In the case $\gamma<\alpha$ of Theorem \ref{teogamma} we add the condition $w\in A_{p,q,\text{loc}}$ to get the sufficiency. Working as in Lemma  \ref{lema61} we can prove that $w\in A_{p,q,\text{loc}}$ implies \eqref{frmaxext}, but we do not know whether the opposite holds for $\gamma<\alpha$. This is why we do not get a full characterization as was the case for $\gamma=\alpha$. 

\section{Power weights}\label{zortzi}

In this section we obtain sharp power weighted inequalities in the case of the Samko and Komori-Shirai versions of the weighted local Morrey spaces. We use the notation $w_\beta(x)=|x|^\beta$. 

\subsection{Weighted spaces of Samko type}

The Samko type spaces are defined with the condition $\varphi_\lambda (B)=r_B^\lambda$ with $\lambda>0$. First we give a lemma.

\begin{lemma}\label{lema71}
Let $1\le p< p_0<\infty$ and $\lambda,\mu>0$. Let $\beta$ and $\delta$ be related by 
\begin{equation}\label{condsubset}
\beta + \frac{n-\lambda}p = \delta + \frac{n-\mu}{p_0}.
\end{equation}
Then $\|f\|_{\mathcal{LM}^p(\varphi_{\lambda}, w_{\beta}^p)}\lesssim \|f\|_{\mathcal{LM}^{p_0}(\varphi_{\mu}, w_{\delta}^{p_0})}$.
\end{lemma}

\begin{proof}
We can use \eqref{normrestrict} for the norm in the Morrey spaces. Let $B$ be a ball such that $|c_B|=4 r_B$. Then we have $r_{\widetilde B}\sim |c_B| \sim |x|$ for $x\in B$. Using these equivalences and H\"older's inequality we get
\begin{equation*}
\aligned
& \left(\frac 1{r_{\widetilde B}^\lambda}\int_{B}|f(x)|^p  |x|^{\beta p}dx\right)^{1/p}  \lesssim |c_B|^{-\frac \lambda p + \beta + \frac np-\frac n{p_0}} \left(\int_{B}|f|^{p_0} \right)^{1/p_0}\\
&\qquad\qquad\lesssim |c_B|^{-\frac \lambda p + \beta + \frac np-\frac n{p_0}+\frac \mu{p_0}-\delta} \left(\frac 1{r_{\widetilde B}^{\mu}}\int_{B}|f(x)|^{p_0}|x|^{\delta p_0}dx \right)^{1/p_0}\\
&\qquad\qquad\le \|f\|_{\mathcal{LM}^{p_0}(\varphi_{\mu},w_{\delta}^{p_0})},
\endaligned
\end{equation*}
where we used \eqref{condsubset}. The result follows.
\end{proof}

\begin{proposition}\label{powernec}
 Let $\varphi_\lambda (B)=r_B^\lambda$ with $\lambda>0$. If $M_\alpha$ is bounded from $\mathcal{LM}^p(\varphi_{\lambda_1}, w_{\beta_1}^p)$ to $W\mathcal{LM}^q(\varphi_{\lambda_2},w_{\beta_2}^q)$, then 
 \begin{equation}\label{necpowers}
\beta_2-\beta_1- \frac{n-\lambda_1}p + \frac{n-\lambda_2}q+\alpha=0.
\end{equation}
Moreover,
 \begin{equation}\label{necpowerange}
\alpha\le \beta_1 +\frac{n-\lambda_1}p  < n.
\end{equation}

If $I_\alpha$ is bounded from $\mathcal{LM}^p(\varphi_{\lambda_1}, w_{\beta_1}^p)$ to $W\mathcal{LM}^q(\varphi_{\lambda_2},w_{\beta_2}^q)$ the same conditions are necessary, and moreover the left-hand side inequality in \eqref{necpowerange} is strict.
\end{proposition}

Using \eqref{necpowers} we can write \eqref{necpowerange} in the equivalent form
 \begin{equation*}
0\le \beta_2 +\frac{n-\lambda_2}q  < n-\alpha.
\end{equation*}
This implies that $\beta_2 q\ge \lambda_2-n$, and from \eqref{necpowerange} we also have $\beta_1 p>\lambda_1-n$.

\begin{proof}
Let $B$ be such that $4r_B\le |c_B|$. Then
\begin{equation*}
\|\chi_B\|_{\mathcal{LM}^p(\varphi_{\lambda_1}, w_{\beta_1}^p)}\sim \left(\frac{w_{\beta_1}^p(B)}{\varphi_{\lambda_1}(\widetilde B)}\right)^{1/p}\sim \left(\frac{|c_B|^{\beta_1 p}r_B^n}{|c_B|^{\lambda_1}}\right)^{1/p}= |c_B|^{\beta_1 - \lambda_1/p} r_B^{n/p}.
\end{equation*}
Similarly,
\begin{equation*}
\|r_B^\alpha \chi_B\|_{W\mathcal{LM}^q(\varphi_{\lambda_2}, w_{\beta_2}^q)}=r_B^\alpha \| \chi_B\|_{\mathcal{LM}^q(\varphi_{\lambda_2}, w_{\beta_2}^q)}\sim  |c_B|^{\beta_2 - \lambda_2/q} r_B^{\alpha + n/q}.
\end{equation*}
Using $r_B^\alpha \chi_B \le M_\alpha \chi_B$ we have 
\begin{equation*}
|c_B|^{\beta_2 - \lambda_2/q} r_B^{\alpha + n/q}\lesssim |c_B|^{\beta_1 - \lambda_1/p} r_B^{n/p},
\end{equation*}
that is,
\begin{equation*}
|c_B|^{\beta_2 - \lambda_2/q-\beta_1 + \lambda_1/p} r_B^{\alpha + n/q-n/p}\lesssim 1.
\end{equation*}
For fixed $c_B$ letting $r_B$ tend to zero we obtain the condition \eqref{necrange}. Take now $r_B=|c_B|/4$ and we obtain
\begin{equation*}
|c_B|^{\beta_2 - \lambda_2/q-\beta_1 + \lambda_1/p+\alpha + n/q-n/p}\lesssim 1.
\end{equation*}
Since $|c_B|$ can be any number in $(0,\infty)$ we deduce the necessity of \eqref{necpowers}. 

An alternative proof of \eqref{necpowers} can be obtained using a scale argument.

The necessity of the first inequality in \eqref{necpowerange} is deduced from the fact that the characteristic function of a ball centered at the origin is not in $W\mathcal{LM}^q(\varphi_{\lambda_2},w_{\beta_2}^q)$ for $\beta_2 q+ n-\lambda_2< 0$ (see \cite[Proposition 5.3]{DR19}).  On the other hand, the function $|x|^{-n}\chi_{B(0,1)}$ is in $\mathcal{LM}^p(\varphi_{\lambda_1}, w_{\beta_1}^p)$ for $\beta_1 +(n-\lambda_1)/p  \ge n$, and it is not locally integrable, hence the necessity of the second inequality in \eqref{necpowerange} follows.

From the pointwise  inequality $M_\alpha f(x)\le I_\alpha (|f|) (x)$ it is clear that the conditions are necessary for $I_\alpha$. We next show that the equality in \eqref{necpowerange} is not possible for $I_\alpha$.

Assume that $\beta_1 p + n-\lambda_1 =\alpha p$ (equivalently, $\beta_2 q + n-\lambda_2=0$ due to  \eqref{necpowers}). Take $f(y)=|y|^{-\alpha} \chi_{B(0,1)}$. Then $f\in {\mathcal{LM}^p(\varphi_{\lambda_1}, w_{\beta_1}^p)}$. For $|x|<1/2$, 
\begin{equation*}
I_\alpha f(x) \ge \int_{2|x|<|y|<1}\frac {|y|^{-\alpha}}{|x-y|^{n-\alpha}}\, dy \gtrsim \int_{2|x|<|y|<1}|y|^{-n}\, dy\sim -\log (2|x|).
\end{equation*}
The function $-\log (2|x|)\chi_{|x|<1/2}$ does not belong to $W\mathcal{LM}^q(\varphi_{\lambda_2}, w_{\beta_2}^q)$ for $\beta_2 q + n-\lambda_2=0$. This rules out the equality case in \eqref{necpowerange} for $I_\alpha$.
\end{proof}

\begin{theorem}\label{teopot}
Let $0< \alpha<n$, $0\le \gamma \le \alpha$, $1\le p<n/\gamma$, and $1/p-1/q= \gamma/n$. Set $\varphi_\lambda (B)=r_B^\lambda$ for $\lambda>0$. 
\begin{enumerate}
\item Let $1<p<n/\gamma$. Then $M_\alpha$ is bounded from $\mathcal{LM}^p(\varphi_{\lambda_1}, w_{\beta_1})$ to $\mathcal{LM}^q(\varphi_{\lambda_2},w_{\beta_2})$ if and only if \eqref{necpowers} and \eqref{necpowerange} hold. If $p=1$ and $q=n/(n-\gamma)$ in \eqref{necpowers} and \eqref{necpowerange}, then  $M_\alpha$ is bounded from $\mathcal{LM}^1(\varphi_{\lambda_1}, w_{\beta_1})$ to $W\mathcal{LM}^{n/(n-\gamma)}(\varphi_{\lambda_2},w_{\beta_2})$.
\item The same result holds for $I_\alpha$ instead of $M_\alpha$ except that in the left-hand side condition of \eqref{necpowerange} the inequality is strict.
\end{enumerate}
\end{theorem}

\begin{proof}
The necessity of the conditions have been proved in Proposition \ref{powernec}. Let us check the sufficiency. 

\textit{Case $\gamma=\alpha$ and $\beta_1=\beta_2:=\beta$}. Then \eqref{necpowers} reduces to $\lambda_1/p=\lambda_2/q$. 
Using Theorem \ref{teosuffrac1} we need to check \eqref{suffracmax1} for $\varphi(B)=r_B^{\lambda_1/p}=r_B^{\lambda_2/q}$. We only need to verify it for balls centered at the origin, because $w_\beta\in A_{p,q,\text{loc}}$.
 
 Let $B$ be a ball centered at the origin. Then 
\begin{equation*}
\|\chi_B\|_{\mathcal{LM}^q(\varphi_{\lambda_2},w_{\beta}^q)}\sim r_B^{\frac{n-\lambda_2}q+\beta}
\end{equation*}
and
\begin{equation}\label{estcent}
\|\chi_B\|_{\mathcal{LM}^p(\varphi_{\lambda_1}, w_{\beta}^p)'}\lesssim r_B^{n-\frac{n-\lambda_1}p-\beta}.
\end{equation}
The last one is similar to \cite[Lemma 2.12 (b)]{DR20}. Consequently,  \eqref{suffracmax1} holds for balls centered at the origin.

For $I_\alpha$ we need to check \eqref{suffracint1}. By Remark \ref{rem65} we only need to consider $B$ centered at the origin. We use that	
\begin{equation*}
M\chi_B(x) \sim \chi_B(x) + \sum_{k=1}^\infty 2^{-kn} \chi_{2^kB\setminus 2^{k-1}B}(x).
\end{equation*}
With \eqref{estcent} for $2^kB$ we have  
\begin{equation*}
\|(M\chi_B)^{1-\alpha/n}\|_{\mathcal{LM}^p(\varphi_{\lambda_1}, w_{\beta}^p)'} \lesssim \sum_{k=0}^\infty 2^{-k(n-\alpha)} (2^k r_B)^{n-\frac{n-\lambda_1}p-\beta} 
\lesssim r_B^{n-\frac{n-\lambda_1}p-\beta},
\end{equation*}
where in the last step we use that the exponent of $2^k$ is negative due to the strict inequality assumed in \eqref{necpowerange}. Consequently, \eqref{suffracint1} is satisfied.

\textit{Case $\gamma=\alpha$ and $\beta_2\ne \beta_1$}. If $\beta_2>\beta_1$, let $\lambda_1'=\lambda_1+(\beta_2-\beta_1)p$. Then using Proposition \ref{rm23} we have
\begin{equation*}
\mathcal{LM}^p(\varphi_{\lambda_1}, w_{\beta_1}^p) = \mathcal{LM}^p(\varphi_{\lambda_1'}, w_{\beta_2}^p).
\end{equation*}
Condition \eqref{necpowers} becomes $\lambda_1'/p=\lambda_2/q$ and \eqref{necpowerange} is satisfied with $\beta_2$ and $\lambda_1'$ instead of $\beta_1$ and $\lambda_1$. Thus we can apply the result in the first part of the proof to deduce the boundedness from $\mathcal{LM}^p(\varphi_{\lambda_1'}, w_{\beta_2}^p)$ to $\mathcal{LM}^q(\varphi_{\lambda_2}, w_{\beta_2}^q)$, which is the desired result.

The proof for $\beta_1>\beta_2$ is similar, setting $\lambda_2'=\lambda_2+(\beta_1-\beta_2)q$ and taking into account that
\begin{equation*}
\mathcal{LM}^q(\varphi_{\lambda_2}, w_{\beta_2}^q) = \mathcal{LM}^q(\varphi_{\lambda_2'}, w_{\beta_1}^q).
\end{equation*}

\textit{Case $\gamma<\alpha$.} Although it would be possible to use Theorem \ref{teogamma} we prefer to proceed 
using Lemma \ref{lema71} to provide a different insight.

Given $p,q,\lambda_1,\lambda_2,\beta_1$ and $\beta_2$ such that \eqref{necpowers} and \eqref{necpowerange} hold, we choose $p_0, q_0, \mu_1, \mu_2, \delta_1$ and $\delta_2$ such that 
\begin{gather*}
p\ge p_0, \ q\le q_0, \ \frac 1{p_0}-\frac 1{q_0}=\frac \alpha n,\ \mu_1, \mu_2>0,\\
\beta_1 + \frac{n-\lambda_1}p = \delta_1 + \frac{n-\mu_1}{p_0}, \text{ and } \beta_2 + \frac{n-\lambda_2}q = \delta_2 + \frac{n-\mu_2}{q_0}.
\end{gather*}
Then the corresponding conditions \eqref{necpowers} and \eqref{necpowerange} hold for $p_0, q_0, \mu_1, \mu_2, \delta_1$ and $\delta_2$, and $p_0$ and $q_0$ satisfy the conditions in the previous part of the proof. Using Lemma \ref{lema71} we obtain
\begin{equation*}
\aligned
\|M_\alpha f\|_{\mathcal{LM}^q(\varphi_{\lambda_2}, w_{\beta_2}^q)}&\lesssim 
\|M_\alpha f\|_{\mathcal{LM}^{q_0}(\varphi_{\mu_2}, w_{\delta_2}^{q_0})}\\
 & \lesssim \|f\|_{\mathcal{LM}^{p_0}(\varphi_{\mu_1}, w_{\delta_1}^{p_0})}\lesssim
\|f\|_{\mathcal{LM}^p(\varphi_{\lambda_1}, w_{\beta_1}^p)}.
\endaligned
\end{equation*}

In the case $p=1$ one has to work with weak-type inequalities, but the proof has only small changes and we do not write the details here.
\end{proof}

The result for $n/p-n/q=\alpha$ and $\beta_1=\beta_2$ is in \cite{NST19}. When $\beta_1=\beta_2=0$ this is known as Spanne's result (usually stated for global Morrey spaces).

\subsection{Weighted spaces of Komori-Shirai type}

Given a weight $w$, let us denote $\psi_{w,\lambda}(B)=w(B)^{\lambda/n}$. The weighted local Morrey spaces of Komori-Shirai type are defined by $L\mathcal M^{p}(\psi_{w,\lambda},w)$ with $\lambda>0$. The typical estimates for fractional operators are not exactly between two spaces of Komori-Shirai type. For instance, when we say that $M_\alpha$ and $I_\alpha$ are bounded from $L\mathcal M^{p}(\psi_{w^q,\lambda},w^p)$ to $L\mathcal M^{q}(\psi_{w^q,\lambda},w^q)$ (see \cite{KS09} for global Morrey spaces) the first space is not of Komori-Shirai type as defined above. To agree with the definition it should be  $L\mathcal M^{p}(\psi_{w^p,\lambda},w^p)$.

Nevertheless, in the case of power weights, the results can be stated in such a way that both spaces are of Komori-Shirai type. This holds because those spaces can be represented as Samko-type spaces as shown in the following proposition.

\begin{proposition}
Let $\lambda>0$ and $\beta>-n$. Then 
\begin{equation}\label{kss}
\mathcal{LM}^{p}(\psi_{w_\beta,\lambda},w_\beta) = \mathcal{LM}^{p}(\varphi_{\lambda(1+\beta/n)},w_\beta) = \mathcal{LM}^p(\varphi_{\lambda},w_{\beta(1-\lambda/n)}).
\end{equation}
A similar result holds for weak-type spaces.
\end{proposition}

\begin{proof}
In all cases we can use the restriction of the Morrey norm to \eqref{normrestrict}. Notice that for the balls in \eqref{normrestrict} we have the equivalence $|c_B|\sim r_B \sim |x|$ for $x\in B$. Then $w_\beta(B)\sim r_B^{\beta+n}$ and the result follows.
\end{proof}

Using this equality of spaces we can rewrite Theorem  \ref{teopot} in terms of Komori-Shirai type spaces. We state it as a corollary.

\begin{corollary}
Let $0< \alpha<n$, $0\le \gamma \le \alpha$, and $1/p-1/q= \gamma/n$. 
 Let $1<p<n/\gamma$. Then  $I_\alpha$ is bounded from $\mathcal{LM}^p(\psi_{w_{\beta_1}^p,\lambda_1},w_{\beta_1}^p)$ to $\mathcal{LM}^q(\psi_{w_{\beta_2}^q,\lambda_2},w_{\beta_2}^q)$ if and only if 
 \begin{equation}\label{condksp1}
(n-\lambda_2)\left(\frac 1q+\frac{\beta_2}n\right) -(n-\lambda_1)\left(\frac 1p+\frac{\beta_1}n\right)+\alpha=0.
\end{equation}
and
 \begin{equation}\label{condksp2}
\alpha <(n-\lambda_1)\left(\frac 1p+\frac{\beta_1}n\right) < n.
\end{equation}
\par 
With $p=1$ and $q=n/(n-\gamma)$ the above conditions are necessary and sufficient for $I_\alpha$ to be bounded from $\mathcal{LM}^1(\psi_{w_{\beta_1},\lambda_1},w_{\beta_1})$ to $W\mathcal{LM}^{n/(n-\gamma)}(\psi_{w_{\beta_2}^{n/(n-\gamma)},\lambda_2},w_{\beta_2}^{n/(n-\gamma)})$.

Similar results hold for $M_\alpha$ adding the equality at the left-hand side of \eqref{condksp2}. 
\end{corollary}

Notice that for the local integrability of $w_{\beta_1}^p$ and $w_{\beta_2}^q$ we need $\beta_1p+n>0$ and $\beta_2q+n>0$. From \eqref{condksp2} we deduce $\lambda_1<n$ and together with \eqref{condksp1} we obtain the necessity of $\lambda_2<n$. For  $M_\alpha$ in the equality case of the left-hand side of \eqref{condksp2}  $\lambda_2=n$ is admissible too.



\begin{thebibliography}{99}


\bibitem{A96}  Alvarez, J.:
Continuity of Calder\'on-Zygmund type operators on the predual of a Morrey space. Clifford algebras in analysis and related topics (Fayetteville, AR, 1993), 309--319,
Stud. Adv. Math., CRC, Boca Raton, FL (1996)

\bibitem{DR19} Duoandikoetxea, J.; Rosenthal, M.:
Boundedness of operators on certain weighted Morrey spaces beyond the Muckenhoupt range. Potential Anal. \textbf{53}, 1255--1268 (2020)

\bibitem{DR20} Duoandikoetxea, J.; Rosenthal, M.:
Boundedness properties in a family of weighted Morrey spaces with emphasis on power weights. J. Funct. Anal.  \textbf{279}, no. 8, 108687, 26 pp. (2020)

\bibitem{DR21} Duoandikoetxea, J.; Rosenthal, M.:
Muckenhoupt-type conditions on weighted Morrey spaces. J. Fourier Anal. Appl. \textbf{27}, no. 2, Paper No. 32, 33 pp.  (2021)

\bibitem{Gr14} Grafakos, L.: Modern Fourier Analysis. Springer Verlag (2014)

\bibitem{HSV14} Harboure, E.; Salinas, O.; Viviani, B.:
Local maximal function and weights in a general setting.
Math. Ann. \textbf{358}, 609--628  (2014)

\bibitem{HSV19} Harboure, E.; Salinas, O.; Viviani, B.:
Local fractional and singular integrals on open subsets.
J. Analyse Math. \textbf{138}, 301--324  (2019)

\bibitem{IKS11} Iida, T.; Komori-Furuya, Y.; Sato, E.: 
The Adams inequality on weighted Morrey spaces. 
Tokyo J. Math. \textbf{34}, 535--545  (2011)

\bibitem{KS09} Komori, Y.; Shirai, S.: Weighted Morrey spaces and a singular integral operator. Math. Nachr. \textbf{282}, 219--231  (2009)


\bibitem{KFS17} Komori-Furuya, Y.; Sato, E.:
Fractional integral operators on central Morrey spaces. 
Math. Inequal. Appl. \textbf{20}, 801--813  (2017)

\bibitem{KFS17b} Komori-Furuya, Y.; Sato, E.: Weighted estimates for fractional integral operators on central Morrey spaces. Math. Nachr. \textbf{290}, 2901--2908 (2017)

\bibitem{H20} Ho, K.-P.:
Definability of singular integral operators on Morrey-Banach spaces. 
J. Math. Soc. Japan \textbf{72}, 155--170  (2020)


\bibitem{LS10} Lin, C.-C.; Stempak, K.: Local Hardy-Littlewood maximal operator. Math. Ann. \textbf{348}, 797--813  (2010) 

\bibitem{MW74} Muckenhoupt, B.; Wheeden, R.: Weighted norm inequalities for fractional integrals. Trans. Amer. Math. Soc. \textbf{192}, 261--274  (1974) 

\bibitem{NST18} Nakamura, S.; Sawano, Y.; Tanaka H.: The fractional operators on weighted Morrey spaces. 
J. Geom. Anal. \textbf{28}, 1502--1524  (2018)

\bibitem{NST19} Nakamura, S.; Sawano, Y.; Tanaka H.: Weighted local Morrey spaces. Ann. Acad. Sci. Fenn. Math., \textbf{45}, 67--93  (2020)

\bibitem{Sam09} Samko, N.: Weighted Hardy and singular operators in Morrey spaces. J. Math. Anal. Appl. {\bfseries 350}, 56--72  (2009)

\bibitem{SFH20} Sawano, Y.; Di Fazio, G.; Hakim, D. I.: \textit{Morrey Spaces.
Introduction and Applications to Integral Operators and PDE's}, vol. I and II. Chapman \& Hall/CRC Monographs and Research Notes in Mathematics, Boca Raton (2020)

\bibitem{SW92} Sawyer E,; Wheeden, R. L.: Weighted inequalities for fractional integrals on Euclidean and homogeneous spaces, Amer. J. Math. \textbf{114}, 813--874  (1992)

\bibitem{ST89} {Str\"omberg, J.-O., Torchinsky, A.}:
\textit{Weighted Hardy spaces}.
Lecture Notes in Mathematics, 1381. Springer-Verlag, Berlin (1989)

\bibitem{W93} Wheeden, R.:  A characterization of some weighted norm inequalities for the fractional maximal function. Studia Math. \textbf{107}, 257--272  (1993)

\end{thebibliography}
\end{document}